\documentclass[12pt]{amsart}
\usepackage{amssymb,enumerate,mathrsfs, amsmath, amsthm, bbm}
\usepackage{url,titletoc,enumitem}
\usepackage[utf8]{inputenc}

\usepackage[usenames,dvipsnames]{xcolor}
\definecolor{darkblue}{rgb}{0.0, 0.0, 0.55}
\usepackage{cmap}

\newtheorem{thm}{Theorem}[section]

\newtheorem{cor}[thm]{Corollary}
\newtheorem{exa}[thm]{Example}
\newtheorem{lem}[thm]{Lemma}
\newtheorem{rmk}[thm]{Remark}
\newtheorem{defi}[thm]{Definition}

\numberwithin{equation}{subsection}

\usepackage{color}
\definecolor{orange}{rgb}{1, 0.55, 0.0}

\def\CR{ \color{red} }

\def\ssec{\subsection}
\def\sssec{\subsubsection}

\def\ben{\begin{enumerate}}
\def\een{\end{enumerate}}

\def\bem{\begin{pmatrix}}
\def\eem{\end{pmatrix}}

\def\beq{\begin{equation}}
\def\eeq{\end{equation}}

\def\alp{{\alpha}}
\def\talp{{\widetilde \alpha}}
\def \talpha {\tilde{\alpha}}
\def \bmone {\mathbbm{1}}

\def \pars{{(s)}}

\def\del{{\delta}}
\def\alp{{\alpha}}

\def\C{{\mathbb{C}}}
\def\R{{\mathbb{R}}}

\def\bs{\bigskip}

\def\onetaut{\bmone^{\tau(\tilde \alpha) }}
\def\onetau{\bmone^{\tau(\alpha)}}

\def\onea{\bmone^\alpha}
\def\onea{\bmone^\alpha}
\def\oneta{\bmone^{\talpha}}
\def\oneb{\bmone^\beta}

\def\nc{noncommutative}

\def\lps{linear product sum}

\def\wdec{Waring decomposition}

\def\tgdd{{S^g_{\delta d} }}

\newcommand{\df}[1]{{\bf{#1}}{\index{#1}}}

\title[Efficient NC Polynomial Evaluation]{Efficient evaluation of noncommutative polynomials using tensor and noncommutative Waring decompositions}

\author[E. Evert]{Eric Evert${}^{1}$}
\address{Eric Evert, Group Science, Engineering and Technology\\
 KU Leuven Kulak, \\
  E. Sabbelaan 53, 8500 Kortrijk, Belgium \\
  and
  \newline
   Electrical Engineering ESAT/STADIUS\\
  KU Leuven, \\
  Kasteelpark Arenberg 10, 3001 Leuven, Belgium
   }
   \email{eric.evert@kuleuven.be}
\thanks{${}^1$Research supported by the NSF grant
DMS-1500835}

\author[J.W. Helton]{J. William Helton${}^1$}
\address{J. William Helton, Department of Mathematics\\
  University of California \\
  San Diego
   }
   \email{helton@math.ucsd.edu }

\author[S. Huang]{Shiyuan Huang${}^1$}
\address{Shiyuan Huang, Department of Computer Science\\
  Columbia University \\
  New York City
   }
   \email{shh029@ucsd.edu}

\author[J. Nie]{Jiawang Nie}
\address{Jiawang Nie, Department of Mathematics\\
  University of California \\
  San Diego
   }
   \email{njw@math.ucsd.edu } 

\date{\today}

\allowdisplaybreaks
\makeindex

\begin{document}

\begin{abstract}

This paper analyses a Waring type decomposition of a noncommuting (NC) polynomial $p$ with respect to the goal of evaluating $p$ efficiently on tuples of matrices. Such a decomposition can reduce the number of matrix multiplications needed to evaluate a noncommutative polynomial and is valuable when a single polynomial must be evaluated on many matrix tuples.

In pursuit of this goal we examine a noncommutative analog of the classical Waring problem and various related decompositions.
For example, we consider a “Waring decomposition” in which each product of linear terms is actually a power of a single linear NC polynomial or more generally a power of a homogeneous NC polynomial.
We describe how NC polynomials compare to commutative ones
with regard to these decompositions, describe a method for computing the NC decompositions and compare the effect of various decompositions on  the speed of evaluation of generic NC polynomials.

\end{abstract}

\maketitle

\section{Introduction}
This paper concerns decompositions of noncommutative polynomials as sums of products of linear polynomials. The goal is to find ways of quickly evaluating noncommutative polynomials on tuples of matrices. 

A place where efficient evaluations matter comes in numerical solution of problems arising in linear systems and control. Problems which are completely specified by signal flow diagrams having  $L^2$ signals all take the form of solving collections of matrix inequalities based on polynomial matrix inequalities. For example, see \cite{CHS06}.

After changes of variables, some basic problems of this type convert to solving Linear Matrix Inequalities (whose coefficients are  functions of the given system parameters)  and for these there are numerous numerical optimization schemes \cite{WSV00}. As with all optimization algorithms these require  very many function evaluations.

\cite{CHS06} showed how, using NC symbolic software, one could produce optimization algorithms whose linear subproblem has coefficients which are NC polynomials in the current iterate $\chi^{(k)}$. As $\chi^{(1)}, \chi^{(2)}, \dots$ progresses toward the optimum, many function evaluations of NC polynomials are required.

The striking fact is that the NC polynomials $p_1, \dots,p_s$ which must be evaluated depend only on the signal flow diagram and on the numerical optimization algorithm in the package. They do not depend on what is being designed, e.g.. a ship controller, airplane controller or helicopter controller (not to mention which ship, which plane etc).

Thus in the lifetime of a popular software toolbox a few specific polynomials must be evaluated billions (at least) of times on matrices of various sizes.

Pursuits involving noncommutative polynomials are in the spirit of the burgeoning area called {\it free analysis}.
Here one takes classical problems and works out analogues with \nc \ variables,
which are {\it free} of constraints.
These free analogues typically have interpretations for matrix or operator variables
and their development often impacts various areas.

One of the original efforts here was Voiculescu's free probability,
which started by developing a notion of entropy for operator variables
and which has a become a big area having many associations to
random matrix theory, \cite{MS17}.
Some other directions are  free
analytic function theory, cf. \cite{KVV14} and   free real algebraic geometry \cite{BKP16}
with some consequences for system engineering being \cite{HMPV09}. Our paper concerns and gives applications for the \nc \ variant of the classical Waring problem.

\subsection{Noncommutative polynomials}
\ \ \ \ 
We work with functions of
 $g$ noncommutative variables
$$x = (x_1, x_2, ..., x_g)$$
and are interested in  powers of
  linear functions
$$L_s(x): = A^\pars_1x_1 + A^\pars_2x_2 + ... A^\pars_g x_g,$$
where $s$ is an index and $A^\pars_i
 \in \mathbb{R} \ or \ \mathbb{C}$ for $1 \le i \le g$.

For any (index) tuple $\alpha = (\alpha_1, \alpha_2, ..., \alpha_d)$, where $\alpha_i $ for $1 \le i \le d$ are  integers between 1 and g, we denote
$$ x^\alpha = x_{\alpha_1}x_{\alpha_2}x_{\alpha_3}...x_{\alpha_d}.$$
 We say the monomial $x^\alpha$ has \df{degree}  $d$.
For example,
if $\alpha = ( 1, 2, 1,3)$, then
$x^\alpha = x_1x_2x_1x_3$ is a degree $4$ monomial.

 A \df{noncommutative (NC) polynomial} is a formal sum of the form
\[
p(x)= \sum_\alpha P_\alpha x^\alpha
\]
where $P_\alpha \in \R$ or $\C$ for each $\alpha$ and only finitely many of the $P_\alpha$ are nonzero. The \df{degree} of a NC polynomial is equal to that of its highest degree monomial which has a nonzero coefficient. If all monomials of a NC polynomial with nonzero coefficients have the same degree, then the NC polynomial is \df{homogeneous}. 

 Let $p(x)=\sum_{|\alpha| \leq d} P_\alpha x^\alpha$ be a noncommutative polynomial in $g$ noncommutative variables. Then for any $n$ and for any $g$-tuple of $n \times n$ matrices $X=(X_1, \dots, X_g)$, we define the \df{evaluation} of $p$ on $X$ by
\[
p(X)=\sum_{|\alpha| \leq d} P_\alpha X^\alpha
\]
where $X^0=I_n$. A question of practical interest is how to efficiently evaluate a NC polynomial on a collection of matrix tuples. 

In this article we show that tensor decompositions may be used to significantly reduce the number of matrix multiplications needed to evaluate a noncommutative polynomial. Here a \df{tensor} is a multiindexed array $T \in (\mathbb C^g)^{\otimes d}$ with entries $T(\alpha) \in \mathbb C$ where $\alpha= (\alpha_1, \dots, \alpha_d)$ is a $d$-tuple of integers between $1$ and $g$.

Our general strategy is as follows. First one associates a homogeneous noncommutative polynomial $p$  to a tensor $T_p \in (\mathbb{C}^g)^{\otimes d}$. By computing the tensor decomposition of the associated tensor, one gets a decomposition that expresses the NC polynomial as a sum of products of linear terms. This reduces the number of matrix multiplications needed to evaluate $p$.

The nonhomogeneous setting can easily be handled can easily be handled by sorting $p$ as a sum of homogeneous polynomials. Additionally, one could homogenize the polynomials with a dummy variable (say $x_0$), then replace $x_0$ with $1$ after a factorization is obtained.

\subsubsection{Evaluation using tensor decompositions}
\label{sec:NCPolyTensor}
Let 
\beq
\label{eq:pForTensorDecomp}
p(x)= \sum_{|\alpha|=d} T(\alpha) x^\alpha
\eeq
 be a homogeneous degree $d$ noncommutative polynomial in $g$ variables $x=(x_1, \dots, x_g)$. We can associate $p$ to the tensor $T_p=(T(\alpha))_{|\alpha|=d}$. Suppose that $T_p$ has a rank $r$ decomposition
\[
T_p= \sum_{s=1}^r A^\pars (1) \otimes A^\pars (2) \otimes \cdots \otimes A^\pars (d)
\qquad 
where \ \  A^\pars (i)=
\bem
 A^\pars_1 (i)\\
  \vdots\\
   A^\pars_g (i)
 \eem   \in \mathbb C^g
\] for each $i$ and $s$. Then we have
\beq
\label{eq:ptgen}
\begin{aligned}
p(x) & = \sum_{s=1}^r \prod_{i=1}^d \Big(A^\pars_1 (i) x_1 + A^\pars_2 (i) x_2 + \cdots + A^\pars_g (i) x_g\Big)\\
&= \sum_{|\alpha|=d} \bigg(  \sum_{s=1}^r \prod_{i=1}^d A^\pars_{\alpha_i} (i) \bigg) x^\alpha.
\end{aligned}
\eeq
We call a decomposition of the form \eqref{eq:ptgen} a \df{\lps } for the NC polynomial $p$. Additionally, if $r$ is as small as possible, we say $p$ has \df{product sum rank} $r$. Before continuing we give an example.

\sssec{Example}
Consider the noncommutative polynomial 
\beq
\label{eq:pLongForm}
\begin{array}{rclcl}
p(x)&=& 20 x_1 x_1 x_1 + 50 x_1 x_2 x_1+20 x_1 x_3 x_1-30 x_2 x_1 x_1-75 x_2 x_2 x_1 \\
&&-30 x_2 x_3 x_1-10 x_3 x_1 x_1 -25 x_3 x_2 x_1- 10 x_3 x_3 x_1 -8 x_1 x_1 x_2 \\
&&-62 x_1 x_2 x_2 -35 x_1 x_3 x_2 +46 x_2 x_1 x_2 +59 x_2 x_2 x_2+10 x_2 x_3 x_2 \\
&& +26 x_3 x_1 x_2+9 x_3 x_2 x_2 -10 x_3 x_3 x_2 +44 x_1 x_1 x_3  +26 x_1 x_2 x_3 \\
&& -10 x_1 x_3 x_3 +2 x_2 x_1 x_3-107 x_2 x_2 x_3-70 x_2 x_3 x_3 +22 x_3 x_1 x_3 \\
&& -57 x_3 x_2 x_3 -50 x_3 x_3 x_3.
\end{array}
\eeq
Think of its coefficients $p_{ijk}$ for $i,j,k =1, 2, 3$ as entries of a tensor $T_p$ with frontal slices
$$
T_p(:,:,1)=\begin{pmatrix} 
20 & 50 & 20 \\
-30 & -75 & -30 \\
-10 & -25 & -10
\end{pmatrix}
\qquad
and
\qquad 
T_p(:,:,2)=\begin{pmatrix} 
-8 & -62 & -35 \\
46 & 59 & 10 \\
26 & 9 & -10
\end{pmatrix}
$$
and
$$
T_p(:,:,3)=\begin{pmatrix} 
44 & 26 & -10 \\
2 & -107 & -70 \\
22 & -57 & -50
\end{pmatrix}
$$
where $T_p(:,:,i)$ is the standard Matlab index notation.
One can check that
$p$ has the rank 2 decomposition
\beq
\label{eq:exTrank2}
T= \begin{pmatrix} 
-3 \\
-4 \\
-4
\end{pmatrix} 
\otimes
\begin{pmatrix} 
-4 \\
4 \\
5
\end{pmatrix} 
\otimes
\begin{pmatrix} 
0 \\
1 \\
2
\end{pmatrix} 
 \ + \ 
\begin{pmatrix} 
-2 \\
3 \\
1
\end{pmatrix} 
\otimes
\begin{pmatrix} 
2 \\
5 \\
2
\end{pmatrix} 
\otimes
\begin{pmatrix} 
-5 \\
5 \\
-5 
\end{pmatrix}.
\eeq
It follows from \eqref{eq:exTrank2} that $p$ has the rank $2$ \lps \  decomposition
\beq
\label{eq:pShortForm}
\begin{array}{rclcl}
p(x)&=& (-3x_1-4x_2-4x_3)(-4x_1+4x_2+5x_3)(x_2+2x_3)\\
& &+(-2x_1+3x_2+x_3)(2x_1+5x_2+2x_3)(-5x_1+5x_2-5x_3)
\end{array}
\eeq
which one can check using NCAlgebra \cite{OHMS17}.

In this case, evaluating $p$ as it is written in equation \eqref{eq:pLongForm} requires $54$ matrix multiplications and $26$ matrix additions. However, using equation \eqref{eq:pShortForm} one needs only $4$ matrix multiplications and $12$ matrix additions, so our complexity is reduced by an order of $10$.

As we illustrate later,  a low rank tensor decomposition
like \eqref{eq:exTrank2} can be computed by standard numerical
software packages such as Tensorlab. Accuracy of the decompositions will be discussed in section \ref{sec:Accuracy}.

\subsubsection{A basic NC Horner method}
\label{sec:NCHorner}

One may also evaluate a NC polynomial using a basic extension of Horner's method to the NC setting. Given a degree $d$ NC polynomial $p(x)$ in $g$ variables, one may first write
\beq
\label{eq:NCHornerStep}
p(x)= c+\sum_{i=1}^g x_i p_i (x)
\eeq
where $c$ is a constant and the degree of $p_i$ is less than $d$ for each $i$. One may then recursively apply this method to each $p_i$ until all polynomials appearing in the summation have degree equal to one. For example, for the polynomial $p(x)$ in equation \eqref{eq:pLongForm}, is equal to
\[
\begin{array}{lclcl}
x_1( x_1 (20 x_1-8x_2+44x_3)+ x_2(50x_1-62x_2+26x_3) + 5x_3(4x_1-7x_3-2x_3))\\
-x_2 ( x_1 (30x_1-46x_2-2x_3)+x_2(75x_1-59x_2+107x_3)+10x_3(3x_1-x_2+7x_3))\\
-x_3 ( x_1 (10x_1-26x_2-22x_3)+x_2(25x_1-9x_2+57x_3)+10x_3(x_1+x_2-5x_3).
\end{array}
\]

Writing $p$ in this form allows $p$ to be evaluated using $12$ matrix multiplications and $26$ matrix additions, thus this method offers a significant improvement over naive evaluation. While this basic Horner method greatly improves on naive evaluation, the linear product sum decomposition for this NC polynomial is still notably more efficient. 

Section \ref{sec:compsave} contains a more detailed comparison of the computational complexity of these three methods for generic homogeneous NC polynomials. We thank a referee for urging us to compare this method to \lps s. Schrempf in \cite{S19} subsequent to this paper introduced an interesting and natural method for evaluation. It heavily uses `linear system realizations', known in the algebra community as `linearization' or the `linearization trick'.

\subsection{Waring decompositions of noncommutative polynomials}
The case where a homogeneous noncommutative polynomial can be expressed as a sum of powers of linear forms adds further advantage for efficient numerical evaluation, as the $d$th power of a matrix can be computed more efficiently than the product of $d$ matrices. This calls for a natural noncommutative generalization of the classical Waring problem. 

The problem is as follows: Given a NC polynomial $p$ in the NC indeterminates $x=(x_1, \dots, x_g)$,  determine if there exist linear functions
\[
L_s(x): = A^\pars_1x_1 + A^\pars_2x_2 + ... + A^\pars_g x_g
\]
such that 
\beq
\label{eq:WarDef}
p(x)= \sum_{s=1}^r (L_s (x))^d
\eeq
where $A_i^\pars \in \R$ or $\C$ for $1=1, \dots, g$. We call a decomposition of the form \eqref{eq:WarDef} a \df{rank $r$ real (resp. complex) Waring decomposition} of $p$. If $r$ is as small as possible then we say $p$ has \df{Waring rank} $r$. 

In the spirit of the NC Waring problem, we also consider the more general problem of determining if a homogeneous NC polynomial of degree $\delta d$ can be decomposed as a sum of $d$th powers of homogeneous degree $\delta$ NC polynomials. That is, supposing $p$ is a homogeneous degree $\delta d$ NC polynomial, we wish to determine if there are homogeneous degree $\delta$ polynomials
\[
G_s (x) = \sum_{|\alpha|=\delta} A_\alpha^\pars x^\alpha
\]
such that
\beq
\label{eq:GenWarDef}
p(x)= \sum_{s=1}^r (G_s (x))^d
\eeq
where each $A_\alpha$ is in $\R$ or $\C$. We call a decomposition of the form \eqref{eq:GenWarDef} a \df{rank $r$ real (resp. complex) $(\delta,d)$-Waring decomposition} of $p$ or sometimes a \df{general Waring decomposition}.

The NC Waring problem reduces to the classical commutative variable Waring problem, thereby effectively solving it over $\C$. In a similar spirit, we reduce the NC general Waring problem to a classical Waring problem, but in more variables, see Section \ref{sec:general}.

\subsection{The NC Waring decomposition}
\ \ \ \
Before stating a result we need a definition.
Define an \df{indicator function} on
 an index $d$-tuple $\alpha=(\alpha_1, \dots, \alpha_d)$ by first defining
\index{$\bmone_j^{\alpha_i}$ }
$$\bmone_j^{\alpha_i}  =
\begin{cases}
1 & \textrm{if }  \alpha_i = j \\
0 & \textrm{if }  \alpha_i \neq j
\end{cases}.
$$
Then the indicator function $\bmone_j^\alpha$ which gives the number of $j$'s appearing in $\alpha$ is
\index{$\bmone_j^{\alpha}$ }
$$
\bmone_j^\alpha := \sum_{i=1}^d \bmone_j^{\alpha_i}.
$$
We caution the reader that the superscript appearing on the indicator function $\bmone_j^\alpha$ is not interpreted as a power.

A corollary for $\del=1$ of Theorem \ref{thm:LinearWaring} is:
\begin{cor}  \label{cor}
Suppose a NC homogeneous polynomial $p(x) = \sum_\alpha {P_\alpha x^\alpha}$, where $P_\alpha = P_{\alpha_1, \alpha_2, ..., \alpha_d}$ $\in \mathbb{C} $,  satisfies $P_\alpha = P_{\talpha}$ for any index sets $\alpha, \talpha$ such that $\bmone_j^\alpha = \mathbbm{1}_j^{\talpha}$  for all $1 \le j \le g$. Then $p$ has a NC complex coefficient \wdec \ with linear powers.
Moreover, for a generic NC homogeneous polynomial,
the number of terms needed is
$$
\left \lceil \frac{\binom{g+d-1}{d}} g  \right \rceil,
$$
except in the cases
\begin{itemize}
     \item $d = 2$, where $g$ terms are needed
     \item $(d,g) = (3,5),(4,3),(4,4),(4,5)$ where $\lceil \frac 1g  {{g+d-1} \choose d} \rceil + 1$ terms are needed.
 \end{itemize}
\end{cor}

\begin{proof}
This corollary is a combination of Theorem \ref{thm:LinearWaring}, the main result in Section \ref{main_result}, and the solutions for the
classical Waring Problem \cite{AH95,OO12}.
\end{proof}

Here the term \df{generic} means that the set of exceptions is contained in a proper closed algebraic variety, i.e., in the zero set of a nontrivial system of polynomial equations.

Each term in a Waring decomposition of a NC polynomial can be evaluated by computing the $d$th power of a matrix rather than computing the product of $d$ different matrices. This gives Waring decompositions an additional computational advantage over \lps \  decompositions when the number of terms needed for each decomposition is the same as one typically expects.

The authors thank Ignat Domanov for discussion related to NC Waring decompositions and efficient polynomial evaluations.

\subsection{Guide to readers}
\ \ \ \
In Section \ref{sec:EvalOfPolys} we examine in more detail the use of \lps \ decompositions to evaluate NC polynomials on matrix variables. We then discus computation of NC Waring and \lps \  decompositions. Additionally we estimate the expected computational savings when evaluating a NC polynomial using one of these decompositions and provide timing comparisons for naive evaluation and evaluation using and \lps \   decompositions.

Section \ref{sec:linear} shows that the NC Waring problem reduces to the classical Waring problem. The section begins by introducing a compatibility condition which is necessary for a NC homogeneous polynomial $p$ to have a Waring decomposition. Theorem \ref{thm:LinearWaring} shows that a NC homogeneous polynomial $p$ has a $t$-term Waring decomposition if and only if it satisfies our compatibility condition and its commutative collapse has a $t$-term Waring decomposition.

Section \ref{sec:general} considers the general NC Waring problem. Similar to the $\delta=1$ case, we begin by introducing a general $\delta$-compatibility condition which is necessary for the existence of a $(\delta,d)$-NC Waring decomposition. Theorem \ref{thm:GenNCWaringMain} shows that, under the $\delta$-compatibility condition, the general NC Waring problem is equivalent to a commutative Waring problem for a polynomial with an increased number of variables. We end with Section \ref{sec:NeedExtraVars} which illustrates that an increase in our number of variables is necessary to reduce the general NC \wdec \ to a commutative \wdec .

\section{Accelerating NC polynomial evaluation using tensor and Waring decompositions}
\label{sec:EvalOfPolys}

In this section we will establish a connection between general tensor decompositions and decompositions of noncommutative polynomials. Using this connection we describe how to use tensor decomposition to efficiently evaluate noncommutative polynomials on matrix variables. Having discussed general tensor decompositions in the introduction, we first consider polynomials with a NC Waring decomposition. Next in Section \ref{sec:compsave} we compare the computational cost of using the various decompositions. Also we discus issues of accuracy.

\subsection{NC Waring decompositions and symmetric tensors}
It is well known that the classical polynomial Waring problem is equivalent to the problem of symmetric tensor decomposition. Let $T \in (\mathbb C^{g})^{\otimes d}$ be a symmetric tensor, i.e. a symmetric multiidexed array, with entries $T_\alpha \in \mathbb C$ where $\alpha= (\alpha_1, \dots, \alpha_d)$ is a $d$-tuple of integers between $1$ and $g$. Here symmetric means that for any permutation $\pi \in \mathcal{S}_d$, we have $T_\alpha=T_{\pi(\alpha)}$ where $\pi(\alpha)=(\alpha_{\pi(1)},\dots, \alpha_{\pi(d)})$. We may associate $T$ to a homogeneous degree d polynomial $p_T (x)$ in the commutative variables $X=(X_1, \dots, X_g)$ by setting
\[
p_T (X)= \sum_{|\alpha|=d} T_\alpha X^\alpha.
\]

Suppose $T$ has rank $r$ symmetric tensor decomposition
\[
T= \sum_{s=1}^r A^\pars \otimes \cdots \otimes A^\pars \quad \quad \mathrm{where \ } d \mathrm{\ copies \ of \ } A^\pars \mathrm{\ appear \ in \ each \ tensor \ product.}
\]
Here $A^\pars=(A_1^\pars, \dots, A_g^\pars)^T \in \mathbb C^{g}$ for each $s$. Then it is straightforward to check that
\[
p_T (X) = \sum_{s=1}^r \left( \sum_{i=1}^g A_i^\pars X_i \right)^d.
\]
That is, a rank $r$ symmetric tensor decomposition of $T$ corresponds to a rank $r$ Waring decomposition for $p_T (z)$. By reversing this correspondence one sees that a rank $r$ Waring decomposition for a homogeneous polynomial gives a rank $r$ symmetric tensor decomposition for the associated symmetric tensor. The fact that the tensor corresponding to a NC polynomial with a Waring decomposition is symmetric is a consequence of Theorem \ref{thm:LinearWaring}.

\subsubsection{Numerical computation of NC Waring decompositions} 
We now give an example which computes an NC Waring decomposition by using popular tensor decomposition software. Consider the homogeneous noncommutative polynomial
\[
\begin{array}{rclcl}
p(x)&=& x_1^3-4 x_2^3 -4 x_3^3 + 5 x_1 x_1 x_2+ 5 x_1 x_2 x_1 + 5 x_2 x_1 x_1 \\
& &-3 x_1 x_1 x_3-3 x_1 x_3 x_1-3 x_3 x_1 x_1 +7 x_2 x_2 x_1 + 7 x_2 x_1 x_2 + 7 x_1 x_2 x_2\\
& &-11 x_2 x_2 x_3 - 11 x_2 x_3 x_2 - 11 x_3 x_2 x_2 + 6 x_3 x_3 x_1 + 6 x_3 x_1 x_3 +6 x_1 x_3 x_3\\
& & - 6 x_3 x_3 x_2 -6 x_3 x_2 x_3-6 x_2 x_3 x_3 + x_1 x_2 x_3 +x_1 x_3 x_2 + x_2 x_1 x_3 \\
& &+x_2 x_3 x_1+ x_3 x_1 x_2+x_3 x_2 x_1.
\end{array}
\]
We associate $p(x)$ to the symmetric tensor $T$ defined by its frontal slices
\[
T(:,:,1)=\begin{pmatrix}
1 & 5 & -3 \\
5 & 7 & 1 \\
-3 & 1 & 6
\end{pmatrix}
\quad \quad
\mathrm{and}
\quad \quad
T(:,:,2)=\begin{pmatrix}
5 & 7 & 1 \\
7 & -4 & -11 \\
1 & -11 & -6
\end{pmatrix}
\]
and
\[
T(:,:,3)=\begin{pmatrix}
-3 & 1 & 6 \\
1 & -11 & -6 \\
6 & -6 & -4
\end{pmatrix},
\]

Using Tensorlab\footnote{A matlab script which computes this decomposition using Tensorlab is avaliable on GitHub at https://github.com/NCAlgebra/UserNCNotebooks.} \cite{VDSBL16}  we compute that $T$ is a rank $4$ tensor and has symmetric tensor decomposition
\[
T= v_1 \otimes v_1 \otimes v_1 +v_2 \otimes v_2 \otimes v_2 +v_3 \otimes v_3 \otimes v_3 + v_4 \otimes v_4 \otimes v_4
\]
where
\[
v_1 \approx \begin{pmatrix}
-0.081 \\
0.409 \\
1.890
\end{pmatrix}
\quad \quad
v_2\approx \begin{pmatrix}
3.165 \\
-3.910 \\
-3.654
\end{pmatrix}
\]
and
\[
v_3 \approx \begin{pmatrix}
-3.273 \\
3.727 \\
3.397
\end{pmatrix}
\quad \quad
v_4 \approx \begin{pmatrix}
1.636 \\
1.581 \\
-1.051
\end{pmatrix}.
\]
It follows that $p$ has the rank 4 NC Waring decomposition
\beq
\label{eq:pNCWaringDecompExample}
\begin{array}{rclcl}
p(x)&\approx& (-0.081 x_1 +0.409 x_2 +  1.890 x_3)^3 \\
& & + (3.165 x_1 -3.910 x_2 -3.654 x_3)^3 \\
& & + (-3.273 x_1 + 3.727 x_2 + 3.397 x_3)^3 \\
& & + (1.636 x_1 + 1.581 x_2 -1.051 x_3)^3.
\end{array}
\eeq
This is easy to numerically verify using NCAlgebra \cite{OHMS17}. 

A naive evaluation of $p$ on a matrix tuple using the original definition of $p$ requires $54$ matrix multiplications. In contrast, evaluating $p$ on a matrix tuple using its NC Waring decomposition only requires $8$ matrix multiplications.

\subsection{Computational savings}
\label{sec:compsave}
 \ \ \ \  We now examine the computational costs for each of the the Waring, \lps, and basic Horner  methods for NC polynomial evaluation.
 
\subsubsection{Linear product sum} 
\ \ \ \ 
The maximum rank of a tensor $T \in (\mathbb C^g)^{\otimes d}$ is not known, however it is conjectured \cite{AOP09} that the rank of a generic tensor $T \in (\mathbb C^g)^{\otimes d}$ is equal to
\beq
\label{eq:TensorGenericRank}
\left\lceil \frac{g^d}{dg-d+1} \right\rceil \approx
\frac{g^{d}}{d(g-1)}
\eeq
except in a small number of defective spaces where most commonly one additional term is needed.

Each term in the \lps \  decomposition is an NC monomial of degree $d$ and may be evaluated in $(d-1)$ multiplications. Therefore, if this conjecture holds, then it follows that generic homogeneous noncommutative polynomials of degree $d$ in $g$ variables may be evaluated using approximately
\beq
\label{eq:LPScompcost}
\frac{(d-1)g^d }{d (g-1)}
\eeq
matrix multiplications.

\subsubsection{Waring}
\ \ \ \
We now consider the case where $p$ has a $t$-term degree $d$ NC Waring decomposition as in equation \eqref{eq:WarDef}. In this case, for any matrix tuple $X$ we may evaluate $p(X)$ using $tg-1$ matrix additions and $t$ matrix exponentiations of degree $d$, where for generic NC polynomials $t \leq \lceil \frac 1g  {{g+d-1} \choose d} \rceil + 1$ by Corollary \ref{cor}. We note that powers of a matrix may be efficiently computed either by decomposing the exponent as a sum of powers of two, or by first computing the Jordan form of the matrix.

Using repeated squaring methods, a matrix exponentiation of degree $d$ can be evaluated with at most $2 \lfloor \log_2 d \rfloor$ matrix multiplications. In addition, using Stirling’s approximation one can show that
\[
\left\lceil \frac 1g  {{g+d-1} \choose d} \right\rceil + 1 \lessapprox \frac{1}{g} \left( \frac{e(g+d)}{d} \right)^d.
\]
It follows that if an NC polynomial $p$ has an NC Waring decomposition, then one may evaluate $p$ on matrix variables using approximately
\[
\frac{2 \lfloor \log_2 d \rfloor}{g} \left( \frac{e(g+d)}{d} \right)^d.
\] 
matrix multiplications. Here $e \approx 2.718$. 

\subsubsection{Horner's method}

Using equation \eqref{eq:NCHornerStep}, one sees that if $h(g,d-1)$ denotes the number of matrix multiplications needed to evaluate a degree $d-1$ NC polynomial in $g$ variables using this basic Horner method, then
\[
h(g,d) \leq g(h(g,d-1)+1).
\]
Using $h(g,1)=0$, one then has
\[
h(g,d) \leq \sum_{\ell=1}^{d-1} g^\ell,
\]
with equality for generic NC polynomials.

\subsubsection{The $d=2$ case}

In the case that $p$ is a homogeneous NC polynomial of degree $2$ in $g$ variables, the tensor $T_p$ corresponding to $p$ is in fact a $g \times g$ matrix. It follows that $p$ has \lps \ rank less than or equal to $g$, hence $p$ may be evaluated using at most $g$ matrix multiplications. Horner's method also generically requires $g$ matrix multiplications in the $d=2$ case, while naive evaluation generically requires $g^2$ matrix multiplications.

\subsection{Comparison of computational costs}
\label{sec:compcompare}
 \ \ \ \  In this subsection we compare computational costs for the various methods.

\subsubsection{Comparison of efficiency: Linear product sum vs. Horner}
\ \ \ \ 
We now briefly compare the various methods. Supposing that the generic rank of a tensor $T \in (\C^g)^{\otimes d}$ is in fact given by  equation \eqref{eq:TensorGenericRank} and using the approximation in equation \eqref{eq:LPScompcost}, one finds that for generic homogeneous NC polynomials, the basic Horner method requires approximately
\[
\frac{g^d-g}{g-1} \  -\    (d-1)\left\lceil \frac{g^d}{dg-d+1} \right\rceil \  \approx \ \frac{g^d-dg}{d(g-1)}
\]
more matrix multiplications to evaluate a NC polynomial than the \lps \ method. The above shows that evaluation with linear product sum is more efficient for generic homogeneous NC polynomials than evaluation with Horner for all $(g,d)$ provided $g,d \geq 3$. This increased efficiency leads to a notable improvement for NC polynomials requiring millions of evaluations, see Table \ref{tab:timing}.

While the more practical point is that \lps \ is more efficient than Horner for all fixed $(g,d)$, it gives perspective to look at extremes of ratios. The asymptotic ratio of the number of matrix multiplications needed by \lps \ to that of Horner approaches $1$ as $d$ tends to infinity. Thus, for high degree homogeneous NC polynomials requiring smaller numbers of evaluations, Horner is likely more appropriate due to the computational cost associated with computing a \lps \ decomposition. In contrast, for fixed $d$ this ratio approaches $(d-1)/d$ as $g$ tends to infinity.

\subsubsection{Comparison of efficiency: Waring vs. \lps}
The main advantage of a Waring decomposition  compared to \lps  \  is that a $d$th power of a linear form may be evaluated (by repeated squaring) using no more than  $ 2 \lfloor \log_2 d \rfloor$ matrix multiplications. In contrast, the product of $d$ distinct linear forms naively requires $d-1$ matrix multiplications to evaluate.

The rank of a tensor is necessarily less than or equal to the symmetric rank of a tensor. It follows that if $p$ has a Waring decomposition, hence the corresponding tensor $T$ is symmetric, then the ratio of the number of matrix multiplications needed by the Waring method and the \lps \  method is bounded below by
\beq
\label{eq:WaringVWaringLike}
\frac{ 2 \lfloor \log_2 d \rfloor}{(d-1)}
\eeq
with equality if the rank of $T$ is equal to the symmetric rank of $T$. 

An example of a tensor whose rank is strictly less than its symmetric rank has only recently been produced \cite{S18}. The example is of a symmetric tensor of size $800 \times  800 \times 800$ with rank 903 and symmetric rank greater than 903. The corresponding NC polynomial $p$ is a homogeneous degree $d=3$ polynomial in $g=800$ variables. 

Since the degree of $p$ is $3$, in both the Waring method and the \lps \  method, each monomial requires $2$ matrix multiplications to evaluate. As a consequence, in this example, using the \lps \  decomposition allows for $p$ to be evaluated in strictly fewer matrix multiplications than the Waring decomposition.

Although an example of a tensor with rank less than symmetric rank is known, there are various results showing that rank is equal to symmetric  rank for generic tensors having small rank, e.g. see \cite{COV17,F16}. Additionally, we note that it remains unknown if generic symmetric tensors have rank equal to symmetric rank. 

In the case $d=3$, there is no advantage of using a Waring decomposition over a \lps \ decomposition in terms of number of multiplications required for evaluation. However, we expect that as $d$ grows large, even if there is a gap between the Waring rank and \lps \ rank of a given NC polynomial, a NC Waring decomposition will outperform a \lps\  decomposition in terms of efficiency due to the ability to efficiently evaluate matrix powers.

It is also worth pointing out that the generic symmetric rank for symmetric tensors in $(\C^g)^{\otimes d}$ is strictly less than the generic rank for arbitrary tensors in $(\C^g)^{\otimes d}$ provided $g,d \geq 3$, with the gap becoming increasingly significant as $g$ and $d$ grow. In contrast, Horner's method sees no notable improvement when used on NC polynomials which have a Waring decomposition. Thus both Waring and \lps \ decompositions significantly outperform Horner's method in this setting.

\sssec{Comparison to naive}
All three methods offer a serious improvement over naive evaluation. Since a naive evaluation of a single degree $d$ NC monomial requires $d-1$ matrix multiplications, the naive approach to evaluating a NC polynomial on a matrix tuple generically requires 
\[
g^d (d-1)
\]
 matrix multiplications. It follows that for a NC polynomial with \lps \ rank given by equation \eqref{eq:TensorGenericRank}, the ratio of the number of matrix multiplications used in the \lps \  method to those in the naive method is approximately
\beq
\label{eq:WaringLikeVsNaive}
\frac{(d-1)g^d }{d (g-1)(d-1)g^d}=\frac{1}{d(g-1)}.
\eeq

Similarly, for NC polynomials with a Waring decomposition, the ratio of the number of matrix multiplications used in the Waring method to those in the naive method is then approximately bounded above by
\[
\frac{2 \lfloor \log_2 d \rfloor }{g (d-1)} \left( \frac{e(g+d)}{gd} \right)^d,
\]
a quantity that rapidly approaches zero as $g$ or $d$ increase, provided $3 \leq d,g$.

\subsubsection{Accuracy of computations} 
\label{sec:Accuracy}
\ \ \ \ 
The tensor in Example \ref{eq:pLongForm} has a unique rank $2$ decomposition (up to scaling) which can be shown using Kruskal's condition for uniqueness of tensor decompositions \cite{K77}. Indeed, when the example is treated with Tensorlab a rank $2$ decomposition which is the same (up to scaling) as the decomposition in \eqref{eq:exTrank2} is produced.

For display purposes in equation \eqref{eq:pNCWaringDecompExample} and above we have truncated the coefficients in the decompositions for $T$ and $p$ at the thousandths place which leads to a small round off error. If we use the long form coefficients computed by Tensorlab, then the decomposition for $T$ and $p$ is highly accurate. Note that $T$ has infinitely many rank $4$ tensor decompositions. The computed tensor decomposition depends on the initialization of the algorithm used in the computation.

Although highly accurate decompositions can be computed for small tensors, when working with large tensors of generic rank, one should not expect to exactly compute a tensor decomposition. However, in early steps of noncommutative optimization algorithms, a small amount of error in the computed descent directions is unlikely to cause serious difficulty. Exact evaluations may be used in later steps when near an optimum. Amounts of relative error averaged over our experiments in tensor decompositions for tensors of the various selected $g$ and $d$ are reported in Table \ref{tab:timing}.

\subsubsection{Experiment comparing run times of \lps \ to other methods}

We now give a brief illustration of experimental timing where we for evaluating homogeneous NC polynomials on $20 \times 20$ and $100 \times 100$ matrices using \lps s , Horner, and naive evaluation.

Table \ref{tab:timing} selects several values of $g$ and $d$, in column $1$, and presents properties of the tensor decomposition in the space $(\C^g)^{\otimes d}$ in the last $3$ columns: generic tensor rank, time to find a decomposition, and accuracy of the decomposition. This is the tensor decomposition used for the \lps \ method. 

Columns $2$ and $3$ list how many polynomial evaluations are needed for \lps \ to overcome its tensor decomposition cost, and hence to outperform Horner's method\footnote{The cost of computing a Horner decomposition is assumed to be negligible in this comparison.}. Similarly, columns $4$ and $5$ show when \lps \ breaks even with the naive method\footnote{The estimates are generated as follows: We randomly generate $1000$ pairs of $n \times n$ matrices and compute the average amount of time needed for a single multiplication of a pair $n \times n$ matrices. The number of matrix multiplications needed for a generic rank \lps \  evaluation or a naive evaluation is multiplied by the average amount of time needed for a single matrix multiplication to compute the expected time needed for a single evaluation on $n \times n$ matrices. Using this methodology the average time needed for multiplication of a pair of $20 \times 20$ matrices or $100 \times 100$ matrices was found to be $1.4056*10^{-6}$ seconds or $2.9392* 10^{-5} $ seconds, respectively.}.

\begin{center}
\begin{table}
 \begin{tabular}{ | c | c | c | c | c | c | c | c |} 
 \hline
  & \multicolumn{4}{c|}{no. of  evals. for LPS to   break even vs.} & generic  & \multicolumn{2}{c|}{tensor decomp.}  \\ 
 \cline{2-5} \cline{7-8}
 (g,d) & \multicolumn{2}{c|}{Horner} & \multicolumn{2}{c|}{naive}  & tensor & time & rel. \\
  \cline{2-5}
  & $20 \times 20$ & $100 \times 100$ & $20 \times 20$ & $100 \times 100$ & rank & (s) & error\\
 \hline
 (3,3)\footnotemark & 9,000 &  430  & 409 & 20 & 5 & 0.025 & $1*10^{-14}$ \\ 
 \hline
 (4,4) & 54,751 & 2,618 & 1,856 & 89 & 20 & 1.85 & $8 *10^{-4}$ \\
 \hline
  (8,4) & 139,136 &  6,654 & 1,853 & 89 & 142 &  30.9 & $7 *10^{-4}$ \\
 \hline
 (5,5) & 290,994 & 13,916 & 4,498  &   215  & 149 & 75.3 & $1 *10^{-3}$ \\ 
 \hline
  (3,6) & 171,110 &  8,183  & 3,972 &  190 &  57 & 18.8 & $3 *10^{-4}$  \\
  \hline
\end{tabular}
\caption{ Break even points for evaluation of a homogeneous NC polynomial using LPS to be more efficient than Horner's method or naive evaluation.}
\label{tab:timing}
\end{table}
\end{center}

\footnotetext{The space $(\mathbb C^3)^{\otimes 3}$ is defective and the generic rank for tensors of this size is $5$ rather than the expected $4$.}

In the case that a NC polynomial has low Waring or \lps \  rank, evaluation using these methods will be much more efficient. Also, the tensor decomposition needed to compute the NC polynomial decomposition takes significantly less time to compute and the error in the decomposition will be significantly lower.

\section{The  noncommutative Waring problem} \label{sec:linear}
\ \ \ \

In this section we examine when a noncommutative polynomial has a NC Waring decomposition. Two approaches are considered. First we consider a noncommutative algebra approach. In this approach, we show that if a noncommutative polynomial $p$ has a Waring decomposition, then its coefficients must satisfy a compatibility condition.
If this condition is satisfied, then we prove that $p$ has a $t$-term Waring decomposition if and only if the restriction of $p$ to commuting variables has a classical $t$-term Waring decomposition. 

The second approach makes use of identification of noncommutative polynomials and tensors and known results for tensor decompositions. To an expert in both tensor theory and in NC polynomials the use of this approach and results on NC Waring decompositions may not come as a surprise. However, for our (main) NC polynomial audience we include a self contained NC polynomial proof. 

Before proceeding with proofs we briefly discuss the history of the polynomial Waring problem.

\ssec{History of the Waring decomposition}
The polynomial Waring problem concerns the question whether a given polynomial, $f(x_1, x_2, \dots, x_n)$, can be represented by sums of powers of polynomials, where $x_i$'s are variables which commute.  In this form, the Waring problem is closely related to symmetric tensor decomposition. The polynomial Waring problem for powers of linear forms
was treated successfully in \cite{AH95} and subsequently in \cite{RS00} and \cite{FOS12} and
has been studied extensively, as is shown, for example, in \cite{BC13} and \cite{GV08}.

\subsection{A basic definition}
\ \ \ \
 Noncommutative Waring decompositions are associated with commutative Waring decompositions through a correspondence we now describe.

For a NC polynomial $p$, the associated \df{commutative collapse}, $p_c$, is the commutative polynomial obtained by considering the variables of $p$ to be commutative.
Our notation for commutative collapse for a NC monomial $x^\alpha = x_{\alpha_1} x_{\alpha_2} \dots x_{\alpha_d}$ is $X^\alpha  = X_{\alpha_1} X_{\alpha_2} \dots X_{\alpha_d}$. For example, when $\alpha = (1, 2, 1, 2 )$, $x^\alpha = x_1x_2x_1x_2$ collapses to $X^\alpha = X_1^2X_2^2$.

We impose an equivalence relation $\sim_c$ \index{$\sim_c$} on NC monomials by saying that $x^\alpha$ and $x^{\talpha}$ are \df{commutative equivalent} if they have the same commutative collapse:
$$
x^\alpha \sim_c  x^{\talpha} \quad   \text{iff} \quad
X^\alpha = X^{\talpha}.
$$
Moreover, we say two index tuples $\alpha$ and $\talpha$ are \df{commutative equivalent}, denoted $\alpha \sim_c  {\talpha}$,
iff $x^\alpha \sim_c  x^{\talpha}$.
Note that
$$\alpha \sim_c  {\talpha}\qquad \mbox{ iff} \qquad  \onea_i = \bmone^\talp_i
\ \ \mathrm{for} \ i = 1, \dots, g.$$

\subsection{NC polynomial proof of the NC Waring decomposition}
\ \ \ \
Our presentation contains two parts. First we state a compatibility condition necessary for the existence of a Waring decomposition, \S \ref{sssec:Compatibility}. Second, if the compatibility condition holds,
we reduce the NC Waring problem to the classical commutative
Waring problem, \S \ref{main_result}.

\sssec{The Compatibility Condition}
\label{sssec:Compatibility}
\ \ \ \
As we next  see the following condition is necessary for existence
of a NC \wdec.

\begin{defi}
We say a noncommutative homogeneous degree $d$ polynomial
$$
p(x) = \sum_{|\alpha|=d} {P_\alpha x^\alpha}\qquad P_\alpha := P_{\alpha_1, \alpha_2, ..., \alpha_d} \in \mathbb{R} \ or \ \mathbb{C}
$$
satisfies the \df{compatibility condition} if
\beq
\label{eq:Paequiv}
P_\alpha = P_{\talpha} \qquad for \ all \
\alpha \sim_c \talpha.
\eeq
Sometimes we say that {$p$ is compatible}.
\qed \end{defi}

We note that a noncommutative homogeneous polynomial $p$ satisfies the compatibility condition if and only if the corresponding tensor described in Section \ref{sec:NCPolyTensor} is symmetric. To see this, given a tuple $\alpha=(\alpha_1,\alpha_2, \dots, \alpha_d)$ of length $d$ and a permutation $\pi \in \mathcal{S}_d$ define
\[
\pi (\alpha)=(\alpha_{\pi(1)}, \alpha_{\pi(2)}, \dots, \alpha_{\pi(d)}) \quad \mathrm{and} \quad \pi (x^\alpha)=x^{\pi(\alpha)}.
\]
It is straight forward to check that $x^\alpha \sim_c  x^{\talpha}$ and $\alpha \sim_c \talpha$ if and only if there is a permutation $\pi \in \mathcal{S}_d$ such that $\pi(\alpha)=\talpha$.

Extend the action of $\mathcal{S}_d$ to noncommutative homogeneous polynomials of degree $d$ by
\[
\pi (p(x))=  \sum_{|\alpha|=d} {P_{\pi(\alpha)} x^{\alpha}}.
\]
Then $p$ meets the compatibility condition if and only if
\[
\pi(p(x))=p(x)
\]
for all permutations $\pi \in \mathcal{S}_d$. That is, for all $\alpha$ and all $\pi \in \mathcal{S}_d$, we have $P_\alpha=P_{\pi (\alpha)}$. It follows that the corresponding tensor is symmetric.

The following lemma shows that the compatibility condition is necessary for existence of a NC \wdec.

\begin{lem} \label{NecessaryCond}
If a NC homogeneous polynomial of degree $d$
has a  $t$-term NC \wdec,
then the compatibility condition \eqref{eq:Paequiv} holds. Moreover, if $p$ meets the compatibility condition, then $p$ has a $t$-term NC Waring decomposition over the complex numbers (resp. real numbers) if and only if
\begin{equation} \label {Peq}
P_\alpha = \sum_{s=1}^t \prod_{j=1}^g \left(  A_j^\pars \right)^{\bmone_j^\alpha}
\end{equation}
has a solution $A_j^\pars \in \mathbb{C} ( \textrm{resp. } A_j^\pars \in \mathbb{R} ).$
\end{lem}

\begin {proof}
By definition, $p$ has a $t$-term \wdec \ if and only if
\begin{equation*}
\label{eq:pcore}
\sum_{|\alpha|=d} P_\alpha x^\alpha   = \sum_{s=1}^t [L_s(x)]^d
 =  \sum_{s=1}^t   \sum_{|\alpha|=d} \left( \prod_{i=1}^d A_{\alpha_i}^\pars \right) x^\alpha
= \sum_{|\alpha|=d} \left( \sum_{s=1}^t \prod_{i=1}^d A_{\alpha_i}^\pars \right) x^\alpha.
\end{equation*}
Comparing the coefficients of $x^\alpha$ on both sides, we get
\beq
\label{eq:pwkeyG}
P_\alpha  = \sum_{s=1}^t \prod_{i=1}^d A_{\alpha_i}^\pars
= \sum_{s=1}^t \prod_{j=1}^g \left(  A_j^\pars \right)^{\bmone_j^\alpha}.
\eeq
This also implies $P_\alpha = P_{\talpha}$ if $\bmone_j^\alpha
= \mathbbm{1}_j^{\talpha}$ for all $1 \le j \le g $.
\end {proof}

\begin{exa}
\rm A NC homogeneous polynomial $p(x) = \sum_\alpha {P_\alpha x^\alpha}$ has the complex (resp. real) $2$-term \wdec \
$$
p(x) = (a x_1 + c x_2)^3 + (b x_1 + d x_2)^3
$$
if and only if $p$ is compatible and
\begin{equation}
\begin{aligned}
P_{1,1,1}  &
=   a^3  + b^3
\\
P_{1,1,2}  &= a^2 c + b^2 d = \frac 1 6 ((a+c)^3 +(b+d)^3- (a-c)^3-(b-d)^3)- \frac 1 3 P_{2,2,2}
\\
P_{1,2,2}  &=   a c^2 + b d^2 = \frac 1 6 ((a+c)^3 +(b+d)^3+ (a-c)^3+(b-d)^3) - \frac 1 3 P_{1,1,1}
\\
P_{2,2,2}  &= c^3  + d^3
\end{aligned}
\end{equation}
has a solution $a,b,c,d \in \mathbb{C}$ (resp. $\mathbb{R}$).
\qed
\end{exa}

\sssec{Reduction of NC Waring to Classical Waring}
\label{main_result}
\ \ \ \
We see in this section that the NC Waring problem reduces to the commutative one.

\begin{lem}\label{lem:count_eta}
For an index tuple $\alpha$, denote $\eta[\alpha]$ as the number of $\talpha$'s that satisfy $\onea_j = \oneta_j$  for all $1 \le j \le g$. Then
$$
\eta[\alpha] = \frac {d !} {\prod_{j = 1}^{g} (\bmone_j^\alpha) !}.
$$
\end{lem}

\begin{proof}
The problem is equivalent to calculating how many d-tuples can be formed by elements from
$
\alpha = ( \alpha_1, \alpha_2, \dots, \alpha_d )
$, which is equivalent to
$$
\eta[\alpha] = \frac {\text{\# of permutations of }d\text{ items}} {\text{\# of permutations of repetitions}} =\frac {d !} {\prod_{j = 1}^{g} (\bmone_j^\alpha) !}.
$$
\end{proof}

\begin{thm} \label{thm:LinearWaring}
Suppose $p$ is a homogeneous NC polynomial which satisfies the compatibility conditions \eqref{eq:Paequiv}.
Then the commutative collapse $p_c$ has the \wdec \
\begin{equation} \label{eq:comm_WR}
p_c(X)=\sum_{s=1}^t [ A^\pars_1X_1 + A^\pars_2X_2 + \cdots + A^\pars_g X_g]^d
\end{equation}
(with $X_i$ being commuting variables) if and only if $p$ has the
NC \wdec \
\begin{equation} \label{eq:nc_WR}
p(x)=\sum_{s=1}^t [ A^\pars_1x_1 + A^\pars_2x_2 + \cdots +A^\pars_g x_g]^d .
\end{equation}
Note that the number of terms is the same and the
real coefficients
(resp. complex coefficients) $A_j^\pars$  are the same.
\end{thm}

\def\cR{{\mathcal R}}

\begin{proof}
The proof begins by laying out the algebraic connection
between $p$ and $p_c$. Let $\cR$ denote a set consisting of one  representative from each $\sim_c$ equivalence class. Then from \eqref{eq:Paequiv}, the NC polynomial
$
p(x) = \sum_{|\alpha| = d}  P_\alpha x^\alpha
$
has
commutative collapse satisfying
$$
p_c(X)
 =
 \sum_{\alpha \in \cR  }
\quad  \sum_{ \talpha \sim_c  \alpha } P_{\talpha} X^{\alpha}
 =
 \sum_{\alpha \in \cR }
 P_{c,\alpha} \  X^{\alpha},
$$
where
$
P_{c,\alpha} =
 \sum_{ \talpha \sim_c  \alpha } P_{\talpha}
$.

Thus if $p$ satisfies the compatibility condition
 \eqref{eq:Paequiv}, then
\begin{equation}
 \label{ncToc}
P_{c,\alpha} =  \eta[\alpha ]    P_\talp \qquad \mathrm{for} \ \alp \in \cR
\ \ \mathrm{and} \ \
\alp \sim_c \talp.
\end{equation}
Therefore, $p_c$ is the commutative collapse of a
compatible NC homogeneous degree $d$ polynomial $p$
 iff $P_{c,\alpha} =  \eta[\alpha ]    P_\alpha$ for all index tuples $\alpha \in \cR$ of length $d$.

Now we proceed to prove our theorem. Assume $p$ has the NC \wdec \ \eqref{eq:nc_WR}, we shall obtain a reversible formula for the \wdec \ of $p_c$.
By equation \eqref{ncToc} and Lemma \ref{NecessaryCond}, the commutative collapse $p_c$ is
\begin{equation} \label{eq:collapse}
p_c(X) =
 \sum_{\alpha \in \cR, \  | \alpha | = d  }\eta[\alpha ]    P_\alpha X^\alpha  = \sum_{|\alpha|=d} P_\alpha X^\alpha  = \sum_{|\alpha|=d} \sum_{s=1}^t \prod_{j=1}^g \left( A_j^\pars \right)^{\bmone_j^\alpha} X^\alpha .
\end{equation}
Thus
\begin{equation} \label{eq:collapse2}
p_c(X)
  = \sum_{s=1}^t \sum_{|\alpha|=d} \prod_{i=1}^d A^\pars_{\alpha_i} X^\alpha
  = \sum_{s=1}^t [ A^\pars_1X_1 + A^\pars_2X_2 + ... A^\pars_g X_g]^d .
\end{equation}

On the other hand, suppose $p$'s commutative collapse, $p_c$, has the commutative \wdec \ \eqref{eq:comm_WR}, then
the calculations in \eqref{eq:collapse} and \eqref{eq:collapse2} can be reversed. By comparing coefficients, this is equivalent to
$$
P_{c,\alpha} = \eta[\alpha ]
 \sum_{s=1}^t \prod_{j=1}^g \left( A_j^\pars \right)^{\bmone_j^\alpha}
$$
for all $\alpha \in \cR$.
Therefore by \eqref{ncToc},
$p$ satisfies
 $$
P_\alpha =\sum_{s=1}^t \prod_{j=1}^g \left( A_j^\pars \right)^{\bmone_j^\alpha}
$$
for all index tuples $\alpha$ of length $d$. Hence by Lemma \ref{NecessaryCond}, $p$ has the \wdec \ \eqref{eq:nc_WR}.
Thus under the compatibility condition \eqref{eq:Paequiv}, the NC polynomial $p$ has a \wdec \ iff its commutative collapse $p_c$ has the same \wdec.
\end{proof}

\subsection{NC Waring decompositions and symmetric tensors}

A tensor based approach to the noncommutative Waring problem that can be used to prove Theorem \ref{thm:LinearWaring} is as follows. By considering the correspondence of NC polynomials and tensors described in Section \ref{sec:EvalOfPolys} as well as the relationship between NC polynomial decompositions and tensor decompositions, one sees that a NC polynomial has a NC Waring decomposition if and only if the corresponding tensor has a symmetric tensor decomposition. 

It is well known that a tensor has a symmetric tensor decomposition if and only if the tensor itself is symmetric, e.g. see \cite[Lemma 4.2]{CGLM08} . Therefore, a NC polynomial $p$ has a NC Waring decomposition if and only if the corresponding tensor $T_p$ is symmetric. One may check that the tensor $T_p$ is symmetric if and only if $p$ satisfies the compatibility condition.

\section {The general noncommutative Waring problem} \label{sec:general}
\ \ \ \
 We now consider a more general situation of which the problem in the preceding section is the base case. As you will see, the bookkeeping and notation is formidable, so it is very helpful to have done a simpler case. In the previous section our focus was to determine if a degree $d$ \nc \  homogeneous polynomial can be expressed as sums of powers of linear terms. Now we examine when a degree $\del d$ \nc \  homogeneous polynomial can be expressed as sums of powers of homogeneous degree $\del$ terms.
 
As in the last section, we consider both noncommutative algebra and (for the tensor proficient) tensor based approaches.

\subsection{Classical General Waring Problem.}
\ \ \ \
The classical commutative Waring problem can be generalized from representation by sums of powers of linear functions to representation by sums of powers of homogeneous polynomials. The generalized classical Waring  problem has also been well studied.
According to Theorem 4 in \cite{FOS12},
there is an upper bound for the number of terms needed for
such  problems:

\begin{thm}
A general homogeneous polynomial of degree $\delta d$ in $g$ variables, where $d \ge 2$, can be expressed as a sum of at most
$d^{g - 1}$ $d^{th}$ powers of degree $\delta$ homogeneous
complex coefficient polynomials. Moreover, for a fixed $g$, this bound is sharp for all sufficiently large $\delta$.
\end{thm}

\subsection{Problem formulation and notation}
\ \ \ \
Let \index{$S^g_\delta$} $S^g_\delta$ be the set of all possible $\delta$-tuples whose elements are integers between $1$ and $g$, i.e.,
$$
S^g_\delta = \{ (\alpha^{(1)}, \alpha^{(2)}, \dots, \alpha^{(\delta)}) \mid 1 \le \alpha^{(i)} \le g \}.
$$
Additionally, define \index{$( S^g_\delta )^d $} $( S^g_\delta )^d $ by
$$
( S^g_\delta )^d = \{ (\alpha_1, \alpha_2, \dots, \alpha_d) \mid \alpha_i \in S^g_\delta \}.
$$
That is, $( S^g_\delta )^d$ is the set of $d$-tuples of $\delta$ tuples of indices.
For any $\alpha = (\alpha_1, \dots, \alpha_d) \in (S^g_\delta)^d$, where $\alpha_i = (\alpha_i^{(1)}, \dots, \alpha_i^{(\del)}) \in S^g_\delta, $
we can write
\[
x^{\alpha} = x^{\alpha_1}x^{\alpha_2} \dots x^{\alpha_d}.
\]
That is, $x^\alpha$ is the monomial
\[
x_{\alpha_1^{(1)}} x_{\alpha_1^{(2)}} \dots x_{\alpha_1^{(\delta)}}
\dots x_{\alpha_d^{(1)}}
\dots x_{\alpha_d^{(\delta)}}.
\]

Recall our notation for a degree $\delta$ homogeneous polynomial
$$
H(x) = \sum_{\beta \in S^g_\delta} A_{\beta} x^\beta,
$$
where $A_{\beta} = A_{(\beta^{(1)}, \beta^{(2)}, \dots, \beta^{(\delta)})} \in \mathbb{C}$.

\begin{rmk}
\rm For any $ \alpha = (\alpha_1, \alpha_2, \dots, \alpha_d) \in (S^g_\delta)^d$,  we can identify
$$\alpha = ((\alpha_1^{(1)}, \alpha_1^{(2)}, \dots, \alpha_1^{(\delta)}), \dots, (\alpha_d^{(1)}, \alpha_d^{(2)}, \dots, \alpha_d^{(\delta)}))$$
with
$$
 (\alpha_1^{(1)}, \alpha_1^{(2)}, \dots, \alpha_1^{(\delta)}, \dots, \alpha_d^{(\delta)}) \in S^g_{\delta d}.$$ On the other hand, for any element of $S^g_{\delta d}$, we can reverse this identification and form groups of size $\delta$ to get a $d$-tuple of $\delta$-tuples.
We let $\tau$ \index{$\tau$} denote the bijection
 $$\tau : S^g_{\delta d}  \to (S^g_\delta)^d$$
 which accomplishes this grouping. 
\qed
\end{rmk}

\bs

\noindent
{\bf The} \df{General NC Waring Problem:}

 {\bf
\ \ \ \ Given a NC homogeneous degree $\delta d$ polynomial p, does it have a t-term $d^{th}$ power real NC Waring (resp. complex NC Waring)
decomposition of  degree $\delta$. That is, can $p(x)$ be written as
\beq
\label{eq:deldWaring}
p(x) = \sum_{s = 1}^t (H_s(x))^d = \sum_{s = 1}^t \left ( \sum_{\beta \in S^g_\delta} A^\pars_{\beta} x^{\beta} \right )^d?
\eeq}

\bs

We call this problem the $(\del,d)$-NC Waring problem and say a decomposition of the form \eqref{eq:deldWaring} is a \df{$t$-term $(\del,d)$-NC Waring decomposition}. Similarly for a commutative polynomial $p_c$, we say a decomposition of the form \eqref{eq:deldWaring} (with $x^\beta$ replaced by $X^\beta$) is a \df{$t$-term $(\del,d)$-Waring decomposition}. Note that the problem treated in Section \ref{sec:linear} is exactly the $(1,d)$-NC Waring problem.

An obvious fact is, if $p$ is a degree $\delta d$ NC homogeneous polynomial and $p$ has a $t$-term $(\delta,d)$-NC \wdec ,
then its commutative collapse $p_c$ has a $t$-term $(\delta,d)$-\wdec . For a conjecture on the generic value of $t$ in this commutative case, see \cite[Conjecture 1.2]{LORS19}.

\sssec{Tuple indicator functions}
\ \ \ \
We now extend the notion of indicator function to tuples of $\del$-tuples.
For two $\delta-$tuples $\beta,\gamma \in S^g_{\delta}$, denote
\index{$\bmone_\beta^{\gamma}$ }
$$
\bmone_\beta^{\gamma} =
\begin{cases}
1 & \textrm{if }  \gamma = \beta \\
0 & \textrm{otherwise},
\end{cases}.
$$
Then for an index tuple $\mu \in (S^g_\delta)^d$, the number of times a particular $\delta-$tuple $\beta \in S^g_\delta$ appears in $\mu$ is
$$
\bmone_\beta^\mu := \sum_{k=1}^d \bmone_\beta^{\mu_k}.
$$
Furthermore, denote

\begin{equation}\label{eq:1abi}
\bmone_i^\mu
:= \sum_{\beta \in S^g_\delta, i \in \beta}
\bmone^\mu_\beta
= \sum_{\beta \in S^g_\delta} \bmone^\mu_\beta \oneb_i
\end{equation}
as the number of integers $i$ appearing in all the $\delta$-tuples in $\alpha$.

\subsection{Main results on the general Waring decomposition}
\ \ \ \
Similar to Section 2, we first state a compatibility
condition which is necessary for the existence of a generalized NC \wdec . We then prove that, if this condition holds, then we can reduce the generalized NC Waring problem to a commutative one at the price of increasing our number of variables.

\subsubsection{The Compatibility Condition}
\ \ \ \
The generalized version of the $\delta = 1$ compatibility condition is defined as follows:

\begin{defi}
We say a noncommutative homogeneous polynomial of degree $\delta d$ in g variables of the form

\begin{equation} \label{eq:ncPoly}
p(x) = \sum_{\alpha \in \tgdd} {P_\alpha x^\alpha} \qquad P_\alpha
 \in \mathbb{R} \ \mathrm{or} \ \mathbb{C}
\end{equation}
satisfies the \df {$\delta$-compatibility condition} if
\begin{equation} \label{eq:GeneralPaEquiv}
P_{\alpha} = P_{\talpha}
\end{equation}
for all index sets, $\alpha$, $\talpha \in \tgdd$
such that $\onetau_\beta = \onetaut_\beta \text{ for all }
\beta \in S^g_\delta $.
Consistent with this,
we define the \df{$\delta$-equivalence relation}, denoted $\sim_\delta$ \index{$\sim_\delta$}, on $\tgdd$ by

$$
\alpha \sim_{\delta} \tilde{\alpha}
\quad
\text{iff}
\quad
\onetau_\beta  = \onetaut_\beta
$$
for all $\beta \in S^g_\delta$.
\qed \end{defi}

\begin{rmk}
\label{rem:tupequiv}
\rm Here are a few bookkeeping properties of $\delta$-equivalences.

\ben

\item
\label{it:1EquivIsCommEquiv}
We have $\alpha \sim_1 \talpha$ if and only if $\alpha \sim_c \talpha$.

\item
\label{it:monEquivCompare}
Let $\delta_1,\delta_2 \in \mathbb{N}$ and let $\alpha,\talpha \in S_{\delta_2 d}^g$. If $\delta_2$ divides $\delta_1$, then $\alpha \sim_{\delta_1} \talpha$ implies $\alpha \sim_{\delta_2} \talpha$. In the case where $\delta_2=1$ this follows from equation \eqref{eq:1abi}. The general case is similar.

\item
\label{it:polyEquivCompare}
Let $\delta_1, \delta_2, d \in \mathbb{N}$ and let $p$ be a degree $\delta_1 d$ NC homogeneous polynomial. If $\delta_2$ divides $\delta_1$ and $p$ satisfies the $\delta_2$-compatibility condition then $p$ satisfies the $\delta_1$-compatibility condition.
\een
Items \eqref{it:monEquivCompare} and \eqref{it:polyEquivCompare} highlight that, as $\delta$ grows, it becomes increasingly difficult for fixed monomials $\alpha$ and $\talpha$ of degree divisible by $\delta$ to be $\delta$-equivalent. As an immediate consequence, as $\delta$ grows, it become more likely that a fixed NC homogeneous polynomial $p$ of degree divisible by $\delta$ satisfies the $\delta$-compatibility condition. In the extreme case, monomials $x^\alpha$ and $x^{\talpha}$ of degree $\delta$ are $\delta$-equivalent if and only if $\alpha=\talpha$. As a result, every degree $\delta$ NC homogeneous polynomial satisfies the $\delta$-compatibility condition. \qed

\end{rmk}

\begin{exa}
\rm 
Let
\[
\alpha=(1,2,2,1) \quad \quad \mathrm{and} \quad \quad \talpha=(2,1,1,2).
\]
Then
\[
\alpha \sim_1 \talpha \quad \mathrm{and} \quad  \alpha \sim_2 \talpha \quad  \mathrm{however} \quad  \alpha \not\sim_4 \talpha.
\]
Now let $p$ be the degree four homogeneous NC polynomial
\[
p(x)=x^\alpha+x^{\talpha}=x_1 x_2 x_2 x_1+x_2 x_1 x_1 x_2.
\]
Then $p$ satisfies the $2$-compatibility condition and the $4$-compatibility condition. However, $p$ does not satisfy the $1$-compatibility condition, since the coefficient of $x_1 x_1 x_2 x_2$ in $p$ is $0$ but the coefficient of $x_1 x_2 x_2 x_1$ is $1$ and
\[
x_1 x_1 x_2 x_2 \sim_1 x_1 x_2 x_2 x_1. \qed
\]
\end{exa}

 The following lemma shows that the $\delta$-compatibility condition is necessary for the general NC Waring problem.

\begin{lem}
\label{lem:condgen}
Suppose a NC homogeneous polynomial $p$ of degree
$\delta d$ in $g$ variables has a $t$-term $(\del,d)$-NC \wdec ,
 then $p$  satisfies the $\delta$-compatibility condition. That is, $P_\alp = P_\talp$  if $\alp \sim_\delta \talp$.
 Here $p$ has  coefficients  $P_\alp$.

Moreover, the $(\del,d)$-NC
 Waring problem  has a
solution over the complex numbers (resp. real numbers) if and only if the equation
\begin{equation} \label {eq:pcoef}
P_\alpha = \sum_{s=1}^t \prod_{\beta \in S^g_\delta}
\left (A^\pars_\beta \right)^{\onetau_\beta}
\qquad \alp \in S_{\del d}^g
\end{equation}
has a solution $A_\beta^\pars \in \mathbb{C}$ (resp. $A_\beta^\pars \in \mathbb{R}$).
\end{lem}

\begin{proof}
The polynomial $p$ has a $t$-term $(\delta,d)$-NC \wdec \,
iff $\exists$ $\delta^{th}$ degree  homogeneous polynomials,
$H_1, H_2, \dots, H_t$ satisfying
\begin{align}
\sum_{\alpha \in \tgdd} P_\alpha x^\alpha
& = \sum_{s=1}^t [H_s(x)]^d
  = \sum_{s=1}^t \left [\sum_{\beta \in S^g_\delta}
  A_{\beta}^\pars x^{\beta} \right ]^d\\
& = \sum_{s=1}^t   \sum_{\alpha \in \tgdd }
 \left( \prod_{ \substack { 1 \le j \le d}}
 A_{\tau(\alpha)_j}^\pars x^{\alpha_j}\right) \\
 & = \sum_{\alpha \in \tgdd}
 \left( \sum_{s=1}^t \prod_{ \substack { 1 \le j \le d}}
 A_{\tau(\alpha)_j}^\pars \right) x^\alpha.
 \label{eq:pmain}
\end{align}
\newline
Comparing coefficients we see, equivalent to the $(\delta, d)$-NC Waring
decomposition is:
$$
P_\alpha  = \sum_{s=1}^t \prod_{ \substack {
1 \le j \le d}} A_{\tau(\alpha)_j}^\pars
= \sum_{s=1}^t \prod_{\substack{{j_1, \dots,j_\delta }\\
1 \le j_k \le g}}^g
\left(  A_{(j_1, \dots,j_\delta)}^\pars \right)^{\bmone_{(j_1, \dots,j_\delta)}^{\tau(\alpha)}}
= \sum_{s=1}^t \prod_{\beta \in S^g_\delta} \left (A^\pars_\beta \right)^{\onetau_\beta},
$$
yielding \eqref{eq:pcoef}.

As a consequence
$
P_\alpha = P_{\talpha}
$
for any
$\alpha$
satisfying
$\onetau_\beta = \onetaut_\beta$
for every
$
\beta \in S^g_\delta
$, yielding the first assertion of the theorem.
\end{proof}

\begin{exa}
\rm
Let
$$
    p(x) = (x_1x_2 + x_1^2)(x_2x_1 + x_1^2)
      = x_1x_2^2x_1 + x_1x_2x_1^2 + x_1^2x_2x_1 + x_1^4.
$$
Then $p$ is an example where there is no $(\del,d)=(2,2)$-NC
\wdec; indeed
the $2$-compatibility condition is violated because
$
P_{(1,1,1,2)} = 0 \neq 1 = P_{(1,2,1,1)}
$.
However, its commutative collapse does have the (2,2)-\wdec :
\begin{equation*}
    p_c(X)=X_1^2X_2^2 + 2 X_1^3 X_2 + X_1^4
          = (X_1X_2 + X_1^2)^2.  \qquad \qed
\end{equation*}

\end{exa}

\subsection{Reduction to classical Waring in more variables}
\label{sec:GenWaringReduction}
\ \ \ \

To solve the general $(\delta,d)$-\nc \  Waring  problem we reduce to the  $\delta=1$ case solved by Theorem \ref{thm:LinearWaring}.
This reduction is accomplished by identifying a monomial $x^\beta$
with a new variable $z_\beta$. Namely, fix $\delta$ and define the map $\phi$ on monomials of the form $x^\beta$ for $\beta \in S_\delta^g$ by
$$
\phi (x^\beta) := z_\beta \qquad \mathrm{for \ each \ } \beta \in S_\delta^g
$$
where the $z_\beta$ are noncommutative indeterminates indexed by elements of $S_\delta^g$. 

We extend our definition of $\phi$ to a noncommutative homogeneous polynomial
\[
p(x)=\sum_{\mu \in (S_\delta^g)^d} P_\mu x^{\mu_1} x^{\mu_2} \cdots x^{\mu_d}
\]
 of degree $\delta d$
 by
\beq
\label{eq:phiOnPolys}
\phi(p(x))=\sum_{\mu \in (S_\delta^g)^d} P_\mu \phi(x^{\mu_1}) \phi(x^{\mu_2}) \cdots \phi(x^{\mu_d})=\sum_{\mu \in (S_\delta^g)^d} P_\mu z_{\mu_1} z_{\mu_2}  \cdots z_{\mu_d}.
\eeq

\begin{lem}
\label{lem:phiAlgIso}
The map $\phi$ as defined in equation \eqref{eq:phiOnPolys} defines an algebra isomorphism on the algebra of noncommutative homogeneous polynomials of degree divisible by $\delta$ in the noncommutative indeterminate $x=(x_1,x_2, \dots, x_g)$ which maps to the algebra of noncommutative homogeneous polynomials in the noncommutative indeterminates $\{z_\beta\}_{\beta \in S_\delta^g}$.
\end{lem}
\begin{proof}
This is straightforward from the definition of $\phi$ on a noncommutative homogeneous polynomial of degree $d\delta$.
\end{proof}

Note that in the case of commutative $X$, substitution of $X^\beta$ by a commutative $Z_\beta$ is sometimes used, however, the isomorphism property in Lemma \ref{lem:phiAlgIso} fails, so conclusions are much less precise than what we get here. 

We now give our main result for the $(\delta,d)$-NC Waring problem.

\begin{thm}
\label{thm:GenNCWaringMain}
Let $p$ be a noncommutative homogeneous polynomial of degree $\delta d$ in the indeterminate $x=(x_1, \dots, x_g)$, and let $\phi$ be as defined in equation \eqref{eq:phiOnPolys}. Then we have the following.
\ben
\item
\label{it:GenMain1}
$p(x)$ has a $t$-term  $(\del,d)$-\nc \  Waring decomposition if and only if
$\phi(p(x))$ has a $t$-term $(1,d)$-noncommutative Waring decomposition.

\item
\label{it:GenMain2}
$p(x)$ satisfies the $\del$-compatibility condition if and only if $\phi(p(x))$ satisfies the  $1$-compatibility condition.

\item
\label{it:GenMain3}
$p(x)$ has a $t$-term  $(\del,d)$-\nc \  Waring decomposition if and only if $p(x)$ satisfies the $\del$-compatibility condition  and the commutative collapse of $\phi(p(x))$ has a $t$-term $(1,d)$-Waring decomposition.
\een
\end{thm}
\begin{proof}
To prove item \eqref{it:GenMain1}, assume $p(x)$ has a $t$-term $(\delta,d)$-noncommutative Waring decomposition
\[
p(x)= \sum_{s=1}^t  [\sum_{\beta \in S_\del^g } A_\beta \; x^\beta  ]^d.
\]
By Lemma \ref{lem:phiAlgIso}, $\phi$ is an algebra isomorphism so
\[
\phi(p(x))=\phi\left(\sum_{s=1}^t  [\sum_{\beta \in S_\del^g } A_\beta \; x^\beta  ]^d \right) = \sum_{s=1}^t  [\sum_{\beta \in S_\del^g } A_\beta \; \phi(x^\beta)  ]^d = \sum_{s=1}^t  [\sum_{\beta \in S_\del^g } A_\beta \; z_\beta  ]^d.
\]
This shows $\phi(p(x))$ has a $t$-term $(1,d)$ noncommutative Waring decomposition. The reverse direction follows the same reasoning using $\phi^{-1}$ instead of $\phi$.

To prove item \eqref{it:GenMain2} let
\[
p(x)=\sum_{\mu \in (S_\delta^g)^d} P_\mu x^{\mu_1} x^{\mu_2} \cdots x^{mu_d}.
\]
Then
\[
\phi(p(x))=\sum_{\mu \in (S_\delta^g)^d} P_\mu z_{\mu_1} z_{\mu_2}  \cdots z_{\mu_d}.
\]
Observe
\[
(\mu_1, \dots, \mu_d) \sim_1 (\tilde \mu_1, \dots \tilde \mu_d)
\]
where the $\mu_j$ are viewed as elements of the index set $S_\delta^g$ if and only if
\[
(\mu_1, \dots, \mu_d) \sim_\del (\tilde \mu_1, \dots \tilde \mu_d)
\]
where the $\mu_j$ are viewed as as $\delta$ tuples of elements of $S_\del^g$. It follows that
\[
P_{(\mu_1, \dots, \mu_d)}=P_{(\tilde \mu_1, \dots \tilde \mu_d)} \quad \quad \mathrm{for\ all\ } (\mu_1, \dots, \mu_d) \sim_1 (\tilde \mu_1, \dots \tilde \mu_d).
\]
where the $\mu_j$ are viewed as elements of the index set $S_\delta^g$ if and only if
\[
P_{(\mu_1, \dots, \mu_d)}=P_{(\tilde \mu_1, \dots \tilde \mu_d)} \quad \quad \mathrm{for\ all\ } (\mu_1, \dots, \mu_d) \sim_\del (\tilde \mu_1, \dots \tilde \mu_d).
\]
where the $\mu_j$ are viewed as as $\delta$ tuples of elements of $S_\del^g$.

Item \eqref{it:GenMain3} is an immediate consequence of items \eqref{it:GenMain1} and \eqref{it:GenMain2} with Theorem \ref{thm:LinearWaring}, our main result for $(1,d)$-NC  Waring decompositions.
\end{proof}

\subsection{Additional variables are necessary for the reduction}
\label{sec:NeedExtraVars}
\ \ \ \
It is tempting to try to solve the general $(\del,d)$-NC Waring problem by reducing to the commutative case without introducing additional variables. This section will show that this is not possible.

One may hope that the following are true:

\begin{enumerate}
\item
\label{it:comm_WR_general}
\textit{If $p$ is a degree $\delta d$ NC homogeneous polynomial, which satisfies the $\delta$-compatibility condition \eqref{eq:GeneralPaEquiv},
then its commutative collapse $p_c$ has the \wdec \
\begin{equation} \label{comm_WR_general}
p_c(X) = \sum_{s = 1}^t \left ( \sum_{\beta \in S^g_\delta} A^\pars_{\beta} X^{\beta} \right )^d
\end{equation}
(with $X_i$ being commuting variables) if and only if $p$ has the
NC \wdec \
\begin{equation} \label{nc_WR_general}
p(x)=\sum_{s=1}^t \left ( \sum_{\beta \in S^g_\delta} A^\pars_{\beta} x^{\beta} \right )^d.
\end{equation}
}

\item
\label{it:nc_WR_general}
\textit{The commutative collapse $p_c$ of $p$
has a $t$-term  $(\del,d)$-NC Waring decomposition
iff
the commutative collapse
$\phi(p)_c$ of $\phi(p)$ has a $t$-term
$(1,d)$-NC Waring decomposition.}

\end{enumerate}

The following polynomial gives a counter example to both items. Let
\[
p(x)=x_1^4+x_1 x_2 x_2 x_1 + x_2 x_1 x_1 x_2 +x_2^4
\]
and let $\delta=d=2$. Then $p$ satisfies the $2$-compatibility condition. We will show that the commutative collapse of $p$ has a two term $(2,2)$-\wdec \  but that $p$ does not have a two term $(2,2)$-NC  \wdec.

It is straight forward to check
\[
p_c(X)=X_1^4+2X_1^2 X_2^2 +X_2^4=(X_1^2+X_2^2)^2.
\]
Item \eqref{it:comm_WR_general} would imply that
\[
p(x)=(x_1^2+x_2^2)^2=x_1^4+x_1^2 x_2^2+x_2^2 x_1^2 +x_2^4 \neq x_1^4+x_1 x_2 x_2 x_1 + x_2 x_1 x_1 x_2 +x_2^4 = p(x)
\]
which is contradiction. This shows that item \eqref{it:comm_WR_general} cannot be correct.

In fact, $p$ does not have a two term $(2,2)$-NC  \wdec. To check this set
\[
z_{(1,1)}=x_1 x_1 \quad \quad z_{(1,2)}= x_1 x_2 \quad \quad z_{(2,1)}=x_2 x_1 \quad \quad z_{(2,2)}=x_2 x_2.
\]
Then $\phi(p) (z)=z_{(1,1)}^2+z_{(1,2)} z_{(2,1)} +z_{(2,1)} z_{(1,2)} +z_{(2,2)}^2$ satisfies the $1$-compatibility condition but $\phi(p)$ does not have a two term $(1,2)$-Waring decomposition. To see this, note that the tensor corresponding to $\phi(p)$ is a $4 \times 4$ symmetric matrix which has rank $4$, hence a NC Waring decomposition for $\phi(p)$ requires four terms. 
It follows from Theorem \ref{thm:GenNCWaringMain} \eqref{it:GenMain1} that $p$ does not have a two term $(2,2)$-NC Waring decomposition.

\subsection{General NC Waring and tensors}

Standard tensor techniques can also be used to address the general NC Waring problem and to derive Theorem \ref{thm:GenNCWaringMain}. One may identify the space of NC homogeneous polynomials of degree $\delta d$ with the space of tensors $\mathbb (\mathbb C^g)^{\otimes d \delta} \cong \mathbb ((\mathbb C^g)^{\otimes \delta})^{\otimes d}$. Requiring that a NC polynomial $p$ satisfies the $\delta$-compatibility condition then corresponds to requiring that the corresponding tensor $T_p$ satisfies a restricted symmetry condition. In standard tensor notation one must have $T_p \in S^d ((\mathbb C^g)^{\otimes \delta})$. In words, $T_p$ is a symmetric tensor in the space $V^{\otimes d}$ where $V$ is the space $(\mathbb C^g)^{\otimes \delta}$. The result again follows from the fact that a tensor in $V^{\otimes d}$ has a symmetric tensor decomposition if and only if it is symmetric.

While this is an expedient approach for those familiar with tensor methods, we expect the noncommutative algebra approach to be more clear for NC algebra experts who are not familiar with tensor methods. Furthermore, the tensor based approach does not easily  convert to a condensed statement of Theorem \ref{thm:GenNCWaringMain} which only uses the language of noncommutative polynomials.

\newpage

\tableofcontents

\end{document}